\documentclass[12pt]{amsart}

\usepackage{amssymb}
\usepackage{amsthm}
\usepackage{amsmath}
\usepackage{amsxtra}
\usepackage{mathrsfs}
\usepackage{enumerate}
\usepackage{xcolor}
\usepackage{tikz}
\usepackage{tikz-cd}
\usetikzlibrary{decorations.markings}
\usetikzlibrary{trees}
\usepackage{colonequals}
\usepackage[all]{xy}
\usepackage{rotating}
\usepackage{hyperref}

\numberwithin{equation}{section}

\newtheorem{theorem}[equation]{Theorem}
\newtheorem{corollary}[equation]{Corollary}
\newtheorem{prop}[equation]{Proposition}
\newtheorem{lemma}[equation]{Lemma}
\newtheorem{example}[equation]{Example}

\theoremstyle{remark}
\newtheorem{remark}[equation]{Remark}

\newcommand{\M}{\mathfrak M}

\newcommand{\F}{\mathbb F}

\newcommand{\Q}{\mathbb Q}

\newcommand{\Or}{\mathcal{O}}
\newcommand{\Z}{\mathbb Z}

\newcommand{\p}{\mathfrak{p}}

\newcommand{\sidesim}{\begin{sideways}%
      $\sim$\end{sideways}}

\DeclareMathOperator{\Id}{Id}

\DeclareMathOperator{\nrd}{nrd}

\DeclareMathOperator{\trd}{trd}

\DeclareMathOperator{\nr}{nm}

\DeclareMathOperator{\End}{End}
\DeclareMathOperator{\opchar}{char}

\DeclareMathOperator{\rad}{rad}
\DeclareMathOperator{\Sym}{sym}

\def\det{\operatorname{det}}

\def\M{\operatorname{M}}

\def\det{\operatorname{det}}

\def\GL{\operatorname{GL}}
\def\PGL{\operatorname{PGL}}
\def\M{\operatorname{M}}

\begin{document}

\title{Metacommutation of primes in Eichler orders}
\author{Angelica Babei and Sara Chari}
\address{Department of Mathematics, Vanderbilt University, 1326 Stevenson Center, Station B 407807, Nashville, TN 37240}
\email{angelica.babei@vanderbilt.edu}
\address{Department of Mathematics, Bates College, Hathorn Hall, Lewiston, ME 04240}
\email{schari@bates.edu}
\thanks{The authors would like to thank John Voight for helpful suggestions on how to improve the results in this article.}

\maketitle

\begin{abstract}

In this article, we study the metacommutation problem in locally Eichler orders. From this arises a permutation of the set of locally principal left ideals of a given prime reduced norm. Previous results on the cycle structure were determined for locally maximal orders. As we extend these results, we present an alternative, combinatorial description of the metacommutation permutation as an action on the Bruhat-Tits tree.
\end{abstract}


\section{Introduction}

In this article, we study the metacommutation problem, following work of Conway--Smith \cite{ConwaySmith}, Cohn--Kumar \cite{CohnKumar}, Forsyth--Gurev--Shrima \cite{FGS}, and Chari \cite{SC}. Conway--Smith \cite{ConwaySmith} first proposed the metacommutation problem in the Hurwitz quaternion order $$\Or \colonequals \Z+\Z i + \Z j + \Z \frac{-1+i+j+k}{2}.$$  The problem is as follows. If $\pi$ and $\omega \in \Or$ have distinct prime reduced norms $\nrd(\pi)=p$ and $\nrd(\omega)=q$, we may factor $\pi \omega=\omega'\pi'$, with $\nrd(\pi')=p$ and $\nrd(\omega')=q$, and the choice of $\pi'$ is unique up to left multiplication by units. Define $\sigma_\omega(\pi)=\pi'$ if there is an element $\omega'$ with $\nrd(\omega')=\omega$ such that $$\pi\omega=\omega'\pi'.$$ This is a permutation and is well-defined because of the two unique ways to factor $\pi\omega$ up to units. The study of this occurrence is called the \emph{metacommutation problem}.

The metacommutation problem was first studied by Cohn--Kumar \cite{CohnKumar}. They determined the number of fixed points and the sign of a permutation induced by an element $\omega \in \Or$ and partial results on the cycle structure. Forsyth--Gurev--Shrima \cite{FGS} then viewed the problem as an action of matrix groups on projective space $\mathbb{P}^1(\F_p)$ and showed that in the Hurwitz order, all cycles that are not fixed points have the same length. Chari \cite{SC} generalized the definition of the metacommutation permutation to orders in central simple algebras of arbitrary degree and gave results on the cycle structure in the case where the completion of the order is maximal, using results of Fripertinger \cite{HF}.  In particular, $\sigma_\omega$ is interpreted as a permutation on the set of left ideals of reduced norm $\p$.

In this article, we extend results in the quaternion case to allow the completion to be an Eichler order. Let $R$ be a complete discrete valuation ring with field of fractions $F$ and maximal prime ideal $\p$ generated by an element $p \in R$, and let $\Or=\M_2(R) \cap \gamma^{-1}\M_2(R) \gamma$ be a (local) Eichler $R$-order in the quaternion algebra $B\colonequals \M_2(F)$, where we define $\gamma=\left(
\begin{array}{cc}
0 & 1\\
p^n & 0
\end{array} \right)$. We restrict to permuting the set $\Id(\Or;\p)$ of \emph{principal} left ideals of reduced norm $\p$.  Each ideal  corresponds to a segment in the Bruhat-Tits tree for $\GL_2(F)$, and we interpret $\sigma_\omega$ as an action of $\Or^\times$ on the set of segments associated to $\Id(\Or;\p)$. We use this action to partition $\Id(\Or;\p)$ into two sets, and describe the restriction of $\sigma_\omega$ on each set individually.  

To describe $\sigma_\omega$ further, we define $\Id(\Or;\p)'$ and $\Id(\M_2(R);\p)'$ to be the set of left ideals of $\Or$ and $\M_2(R)$, respectively, having reduced norm $\p$, and omitting one specific ideal to be defined later. For $\omega \in \Or^\times$, define $\sigma_\omega \in \Sym(\Id(\Or;\p))$ by $\sigma_\omega(P)=P\omega$ and define $\tau_\omega \in \Sym(\Id(\M_2(R);\p))$ by $\tau_\omega(P)=P\omega$. Our main result is the following theorem.



\begin{theorem}\label{thm1}
For $\omega \in \Or^\times$, there is a partition $\Id(\Or;\p)'=S_1 \sqcup S_2$ such that $\sigma_\omega|_{S_i}$ permutes $S_i$, and there are bijections $\varphi \colon S_1 \rightarrow \Id(\M_2(R);\p)'$ and $\phi_\gamma \colon S_2 \rightarrow S_1$ such that the following diagrams commute.
\begin{enumerate}

\item[\rm{a.}] $$\xymatrix{
S_2 \ar^{\sigma_\omega|_{S_2}\hspace{.1in}}[r]\ar_{\phi_{\gamma}}@{^{}->}[d]&{S_2}\ar^{\phi_{\gamma}}[d]\\
{S_1}\ar^{\sigma_{\gamma^{-1}\omega \gamma}|_{S_1}\hspace{0in}}[r]&{S_1}
}$$

\item[\rm{b.}] $$\xymatrix{
S_1 \ar^{\sigma_\omega|_{S_1}\hspace{.1in}}[r]\ar_{\varphi}@{^{}->}[d]&{S_1}\ar^{\varphi}[d]\\
{\Id(\M_2(R);\p)'}\ar^{\tau_\omega\hspace{0in}}[r]&{\Id(\M_2(R);\p)'}
}$$

\item[\rm{c.}] $$ \xymatrix{
S_2 \ar^{\sigma_\omega|_{S_2}\hspace{.1in}}[r]\ar_{\varphi \circ \phi_{\gamma}}@{^{}->}[d]&{S_2}\ar^{\varphi \circ \phi_{\gamma}}[d]\\
{\Id(\M_2(R);\p)'}\ar^{\tau_{\gamma^{-1}\omega\gamma}\hspace{0in}}[r]&{\Id(\M_2(R);\p)'}
}$$

\end{enumerate}

\end{theorem}

In this way, we may understand metacommutation in an Eichler order $\Or$ by studying the given permutations in the two maximal orders containing $\Or$. More precisely, up to isomorphism, $\Or \subseteq \M_2(R)$, and it suffices to understand the permutation given by metacommutation in $\M_2(R)$ itself. Our methods allow us to interpret the metacommutation problem in terms of an action on the Bruhat-Tits tree and hence give a combinatorial description of the cycle structure of $\sigma_\omega$ .

\section{Metacommutation setup}
\label{eichlersetup}

We now set up the metacommutation problem and state previous results in this section. Let $R$ be a ring whose field of fractions $F$ is a global field. Let $B$ be a quaternion algebra over $F$ and let $\Or \subseteq B$ be a quaternion $R$-order. For a prime $\p \subseteq R$, denote by $\Or_{\p} \colonequals \Or \otimes_R R_\p$ and $B_{\p} \colonequals B \otimes_F F_\p$ the completions of $\Or$ and $B$, respectively, at $\p$. Define $\Id(\Or;\p)$ to be the set of locally principal left ideals of $\Or$ having reduced norm $\p$.  

We first define the permutation of $\Id(\Or;\p)$ following Chari \cite{SC}. For a left ideal $P$ of reduced norm $\p$, and $\omega \in \Or$ with $\p \nmid \nrd(\omega)$, define  
\begin{equation}\label{permutation} 
\sigma_\omega(P)\colonequals P\omega+\Or \p.
\end{equation} 

Now, it suffices to study the corresponding permutation in the completion $\Or_\p$. The map $P \mapsto P_\p \colonequals P \otimes_R R_\p$ is a bijection between $\Id(\Or;\p)$ and the ideals of reduced norm $\p R_{\p}$ in $\Or_{\p}$ by the local-global dictionary for lattices and \cite[Theorem 5.2(iii)]{IR}. Using this correspondence and given $\omega \in \Or$, we have $$\sigma_\omega(P_\p)=P_\p\omega+\Or_\p \p \colonequals (P_\p)'$$ if and only if $$\sigma_\omega(P)=\sigma_\omega(P_\p \cap \Or)= (P_\p \cap \Or) \omega + \Or \p =(P_\p \omega + \Or_\p \p) \cap \Or=(P_\p)'\cap \Or.$$ We then obtain the same permutation of ideals in $\Or$ and in $\Or_\p$, so we may focus our attention on the principal left ideals of $\Or_\p$ of reduced norm $\p R_\p$.

Furthermore, we will restrict to permuting only those left ideals of $\Or_\p$ that are principal. If a left $\Or_\p$-ideal $\Or_\p \alpha$ is principal, then $\sigma_\omega(\Or_\p \alpha)=\Or_\p \alpha \omega+\Or_\p \p=\Or_\p (\alpha \omega)$ is also principal since $\Or \alpha \omega \subseteq \Or_\p \p$. Therefore, we will still obtain a permutation by restricting the permutation to principal left ideals.

\begin{remark}

In $\Or_\p$, if $P_\p=\Or_\p \alpha$ is principal, we may equivalently define $\sigma_\omega(P_\p)=P_\p\omega=\Or_\p \alpha \omega$. 

\end{remark}

\begin{remark}
For any $\omega \in \Or_{(\p)}^\times$, there is an element $a \in R_{(\p)}^\times$ such that $a \omega \in \Or$, and $P_\p (a\omega)=a P_\p \omega=P_\p \omega$. We may therefore define $\sigma_\omega$ for any $\omega \in \Or_{(\p)}^\times$.
\end{remark}



In the case where $\Or_\p$ is maximal, all left ideals of reduced norm $\p \Or_\p$ are principal as shown in Reiner \cite[Theorem 17.3(iii)]{IR}. If $\Or_\p$ is in fact maximal, results on the cycle structure of the metacommutation permutation are given by the following two theorems. 

\begin{theorem}[Chari]\label{thm2} Given $\omega \in \Or_{\p}^\times$. Define 
\begin{align*}
\tau \colon \GL_m(\F_q)&\rightarrow \Sym(\mathbb{P}^{m-1}(\F_q))\\
\tau(Q)(v)&\mapsto Q^{-1}v
\end{align*} and let $\rho \colon \Or_\p \rightarrow \Or_\p/\rad \Or \simeq \GL_m(\F_q)$ for a finite field $\F_q \supseteq R/\p$. Then, $$\Sym(\Id(\Or_\p;\p)) \simeq \Sym(\mathbb{P}^{m-1}(\F_q))$$ and the following diagram commutes.

$$\xymatrix{
{\Or_{\p}^\times}\ar^{\sigma\hspace{.2in}}[r]\ar_{\rho}@{^{}->}[d]&{\Sym(\Id(\Or_\p;\p))}\ar^{\sidesim}[d]\\
{\GL_m(\F_q)}\ar^{\tau\hspace{.2in}}[r]&{\Sym(\mathbb{P}^{m-1}(\F_q))}
}$$

\end{theorem}

\begin{proof}
See \cite[Theorem 4.7]{SC}. 
\end{proof}

\begin{theorem}[Forsyth--Gurev--Shrima; Chari] \label{thm3}
If $\Or$ is a quaternion algebra, then all cycles of $\sigma_\omega$ that are not fixed points are the same length; i.e., there is an integer $\ell>0$ such that every cycle either has length either 1 or $\ell$.
\end{theorem}

\begin{proof}
See Forsyth--Gurev--Shrima \cite{FGS} and Chari \cite{SC}.
\end{proof}

To set up the metacommutation problem in the Eichler case, we proceed locally. Let $R$ be a complete discrete valuation ring  with maximal ideal $\p=(p)$, residue field $\F_q \colonequals R/\p$, and field of fractions $F$. Let $\Or=\Or_1\cap \Or_2$ be an intersection of two (not necessarily distinct) maximal $R$-orders $\Or_1$ and  $\Or_2$   in the quaternion algebra $B\cong \M_2(F)$ over $F$. Such an order is called a (local) \textit{Eichler} order. By Hijikata \cite{HH},  $\Or$ is conjugate to an order of the form $\begin{pmatrix} R&R \\ \p^n & R\end{pmatrix}$. We call $d(\Or)\colonequals \p^n$ the \textit{level} of the local  Eichler order $\Or$. 

Let $\Id(\Or;\p)$ be   the set of principal left ideals of $\Or$ having reduced norm $\p$. Then we may assume such ideals are generated by elements of reduced norm $p$:

\begin{lemma}
\label{normp}
The set of principal left ideals $\Id(\Or; \p)$ of norm $\p$ is given by \[\Id(\Or; \p)=\{\Or\alpha: \alpha \in \Or, \, \nr(\alpha)=p\}.\]
\end{lemma}

\begin{proof} Let $P \in \Id(\Or; \p)$. Then $P$ is principal, so $P=\Or \alpha$ for some $\alpha \in \Or$. Then $\nr(P)=\nr(\Or)\nr(\alpha)=R\nr(\alpha) $, so $\nr(\alpha)=p u$ for some $u \in R^\times$. Let $\mu \colonequals \begin{pmatrix} u^{-1} &0 \\ 0&1\end{pmatrix} \in \Or^\times$, so $ \Or \mu \alpha = \Or \alpha$ and $\nr(\mu \alpha)=p$, and we have the containment $\Id(\Or; \p) \subseteq \{\Or\alpha: \alpha \in \Or, \, \nr(\alpha)=p\}$. The reverse containment is obvious.
\end{proof} 

Recall that for an element $\omega \in \Or^\times$, we define the permutation $\sigma_\omega$ by
$$\sigma_\omega(P)= P \omega$$ as in \ref{permutation}.
Suppose $\Or=\Or_1 \cap \Or_2$ is not maximal, so $\Or_1$ and $\Or_2$ are distinct maximal $R$-orders. Given $\omega \in \Or^\times$, we will describe the cycle structure of $\sigma_\omega$ in terms of permutations of the left ideals of $\Or_1$ and $\Or_2$ of reduced norm $\p$, separately. When $\rad \Or \in \Id(\Or;\p)$, we show that it is fixed under $\sigma_\omega$ for all $\omega \in (R_\p)^\times$, and instead restrict $\sigma_\omega$ to the set \[\Id(\Or;\p)'\colonequals \Id(\Or;\p) \setminus \rad \Or.\] To accomplish this, in the following section we introduce the \textit{Bruhat-Tits tree} $\mathcal{T}_\p$ for $\GL_2(F)$.

\section{Ideals of reduced norm $\p$}

Eichler \cite{ME} and Brzezinski \cite{JB} show that the cardinality $\#\Id(\Or;\p)$ is either $2q$ or $2q+1$, depending on the level of the order.  In this section, we give a combinatorial description of these ideals, together with their explicit generators.  We retain notation from the end of Section \ref{eichlersetup}, and start with some known facts about the Bruhat-Tits tree for $\GL_2(F)$. For more details about the tree, see Serre \cite{Serre}.  Let $V \colonequals F^2$. We say that two (full) $R$-lattices $L_1$ and $L_2$ in $V$ are \textit{homothetic} if $L_1=L_2 a$ for some $a \in  F^\times$. Homothety of lattices is an equivalence relation, which we will denote by $[L]=[La]$. We can identify $B$ with the ring of linear transformations $\End_F(V)$ acting on $V$ on the right, so  $B^\times$ acts transitively on the lattices in $V$ and preserves homothety: $$[L]\cdot \xi \colonequals [L \xi]=[L\xi a]=[L a\xi]=[L a]\cdot \xi.$$

We now define the Bruhat-Tits tree $\mathcal{T}_\mathfrak{p}$. The set vertices of $\mathcal{T}_\p$ is given by the homothety classes of latices $[L]$ in $V$. There is an edge between two vertices if there are lattices $L_1$ and $L_2$ in their respective homothety classes such that $p L_1 \subsetneq L_2 \subsetneq L_1$.  It is known that $\mathcal{T}_\p$ is a $(q+1)$-regular tree. 

 We can interpret the distance between two vertices via their corresponding homothety classes. In particular, given two lattices $L_1$ and $L_2$, by the invariant factor theorem there exists an $R$-basis $\{f_1, f_2\}$ of $L_1$ such that $\{p^af_1, p^bf_2\}$ is an $R$-basis for $L_2$. Define the distance  $\delta([L_1],[L_2])\colonequals |b-a|$, which is well-defined on homothety classes.  Moreover, the action of $B^\times$ extends to an action on the tree which preserves $\delta$. In Figure \ref{tree}, we see a piece of $\mathcal{T}_2$, the tree for $GL_2(\Q_2)$.
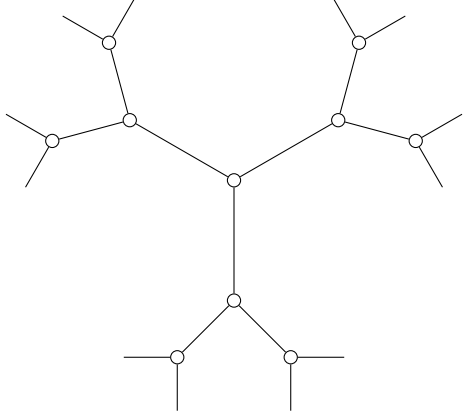
\begin{figure}
\caption{Piece of the  tree $\mathcal{T}_2$. }
\label{tree}
\begin{tikzpicture}[scale=0.8,
  grow cyclic,
  level distance=2cm,
  level/.style={
    level distance/.expanded=\ifnum#1>1 \tikzleveldistance/1.5\else\tikzleveldistance\fi,
    nodes/.expanded={\ifodd#1 fill=none\else fill=none\fi}
  },
  level 1/.style={sibling angle=120},
  level 2/.style={sibling angle=90},
  level 3/.style={sibling angle=90},
  level 4/.style={sibling angle=45},
  nodes={circle,draw,inner sep=+0pt, minimum size=5pt},
  ]
\path[rotate=30]
  node {}
  child foreach \cntI in {1,...,3} {
    node {}
    child foreach \cntII in {1,...,2} { 
      node {}
      child foreach \cntIII in {1,...,2} 
    }
  };
\end{tikzpicture}
\end{figure}
We turn to the connection between the tree $\mathcal{T}_\p$ and Eichler orders in $B$.  We can identify the maximal order $\M_2(R)$ with $\End_R(L_0)$, which acts on the right on  the free  $R$-lattice $L_0 \colonequals Re_1  + Re_2 $, where $\{e_1,e_2\}$ is the standard basis for $V$ . By Reiner \cite[(17.3)]{IR}, each maximal order in $B$ is conjugate to $\M_2(R)$ by an element $\xi \in B^\times$, and we can identify $\xi^{-1}M_2(R)\xi$ with $\End_R(L_0\xi)$. We can also easily see that $[L_1]=[L_2]$ if and only if $\End_R(L_1)=\End_R(L_2)$, so we have bijections between the set of homothety classes of full $R$-lattices in $V$, the set of maximal orders in $B$, and the vertices in $\mathcal{T}_\p$. Moreover, the  action of $B^\times$ on $\mathcal{T}_\p$ determines an action of $B^\times$ on the set of maximal orders corresponding to conjugation.

From now on, we consider non-maximal Eichler orders. Suppose we have an Eichler order $\Or=\Or_1\cap \Or_2$ of level $d(\Or)=\p^n$. 
 Then one can associate to $\Or$ the segment in $\mathcal{T}_\p$ whose endpoints correspond to $\Or_1$ and $\Or_2$, and whose length is precisely $n$. Following from the work of Hijikata \cite{HH}, $\Or$ is also the intersection of all the
maximal orders contained in this segment. Since we can take $\Or$ up to conjugation, from now on we assume $\Or=\begin{pmatrix} R & R \\ \p^n &R\end{pmatrix}$ with $n\ge 1$, in which case
\[ \Or=\bigcap_{i=0}^n 
\begin{pmatrix} 
R & \p^{-i} \\
\p^i&R\end{pmatrix}=\bigcap_{i=0}^n \End_R(L_0\gamma_i)\] where $\gamma_i=\begin{pmatrix} 0&1 \\ p^i&0\end{pmatrix}$. Then each vertex corresponding to $[L_0\gamma_i]$ lies on the segment associated to $\Or$ (see Figure \ref{Osegment}). Denoting $\gamma\colonequals \gamma_n$, we have 
 \begin{equation} \label{eichler} \Or=\Or_1\cap \Or_2=\M_2(R) \cap \gamma^{-1} \M_2(R) \gamma. \end{equation}

\begin{figure}
\caption{Subtree containing the segment associated to $\Or$}
\label{Osegment}

\begin{center}
\begin{tikzpicture}[scale=0.8,
  grow cyclic,
  level distance=2cm,
  level/.style={
    level distance/.expanded=\ifnum#1>1 \tikzleveldistance/1.5\else\tikzleveldistance\fi,
    nodes/.expanded={\ifodd#1 fill=white\else fill=white\fi}
  },
  level 1/.style={sibling angle=12},
  nodes={circle,draw,inner sep=+0pt, minimum size=5pt},
  ]

\draw[dotted] (2,0) -- (4.9,0);

\path[rotate=180]
  node {}
  child foreach \cntI in {1,...,5} {
    node {}
  };

\path[]
  node {}
  child foreach \cntI in {1} {
    node {}
  };

\path[rotate=90] (0,0) -- (0,-2)
  node {}
  child foreach \cntI in {1,...,4} {
    node {}
  };

\path[rotate=90] (0,0) -- (0,-5)
  node {}
  child foreach \cntI in {1,...,4} {
    node {}
  };

\draw (5.1,0) -- (6.9,0);

\path[] (5,0) -- (7,0)
  node {}
  child foreach \cntI in {1,...,5} {
    node {}
  };

\node[draw=none,fill=none,below] at (0,-0.2)  {\fontsize{8}{6}\selectfont$[L_0]$};
\node[draw=none,fill=none,below] at (2,-0.1)  {\fontsize{8}{6}\selectfont$[L_0\gamma_1]$};
\node[draw=none,fill=none,below] at (7,0)  {\fontsize{8}{6}\selectfont$[L_0\gamma_n]$};
\node[draw=none,fill=none,below] at (7,-0.3)  {\fontsize{8}{6}\selectfont$=[L_0\gamma]$};
\node[draw=none,fill=none,below] at (5,0.2)  {\fontsize{8}{6}\selectfont$[L_0\gamma_{n-1}]$};

\end{tikzpicture}
\end{center}
\end{figure}
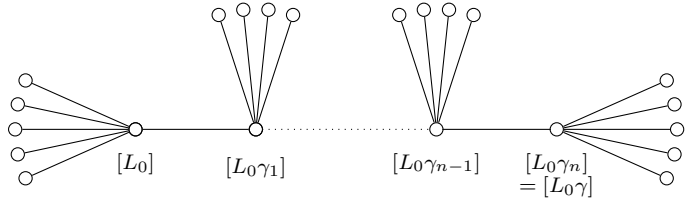

We will make frequent use of the following result:
\begin{lemma}
\label{wfix}
Consider a maximal order $\Lambda \subseteq B$, corresponding to the vertex $X$ in $\mathcal{T}_\p$. Then $\mu \in \Lambda^\times$ acts on $\mathcal{T}_\p$ by fixing $X$.
\end{lemma}
\begin{proof} Suppose $X$ is given by the homothety class $[L]$. Then $\Lambda=\End_R(L)$, so  $L\mu=L$ for $\mu \in \Lambda^\times$. \end{proof}

 Since conjugation by  $B^\times$ gives an action on the tree, any order conjugate to $\Or$ will correspond to a segment of length $n$ as well. We can take the connection between algebraic properties of $\Or$ and the structure of the tree further, and we have the following connection between the set $\Id(\Or; \p)$ and certain segments in the tree.

\begin{lemma}
\label{ideal-order}
Let $\Or$ be an Eichler order. Then there is a bijective correspondence between the sets 
\[ \Id(\Or; \p) \overset{\sim}{\longleftrightarrow} \left\{\alpha^{-1}\Or\alpha: \alpha \in \Or, \, \nr(\alpha)=p\right\}.\] 
\end{lemma}

\begin{proof}
By Lemma \ref{normp}, we have  $\Id(\Or; \p)=\{ \Or \alpha : \alpha \in \Or, \nr(\alpha)=p \}$, so we have a map $\Or\alpha \mapsto \alpha^{-1}\Or \alpha$ and we must show this map is well-defined. Suppose $\Or \alpha = \Or \beta$ for some  elements $\alpha, \beta$ of reduced norm $p$. Then $\Or \alpha \beta^{-1}=\Or$, so $\alpha \beta^{-1} \in \Or^\times$ and therefore  $(\alpha \beta^{-1})^{-1}\Or \alpha \beta^{-1}=\Or$. This implies $\alpha^{-1}\Or \alpha=\beta^{-1}\Or \beta$, and it follows that each ideal has a well-defined  associated Eichler order. 

Conversely, suppose $\alpha^{-1} \Or \alpha=\beta^{-1}\Or \beta$ for two $\alpha, \beta \in \Or$ such that $\nr(\alpha)=\nr(\beta)=p$.  Then $\alpha \beta^{-1} \in \mathcal{N}(\Or)$, where $\mathcal{N}(\Or)$ is the normalizer of $\Or$. Taking $\Or$ as in Equation (\ref{eichler}), by Hijikata \cite{HH}, we have $\mathcal{N}(\Or)=  \gamma F^\times \Or^\times \cup F^\times \Or^\times$. Suppose $\alpha \beta^{-1} \in \gamma F^\times \Or^\times$. Since $R^\times \subseteq \Or^\times$, we can write $\alpha \beta^{-1}= \gamma p^\ell \mu$ for some $\ell \in \Z$ and $\mu \in \Or^\times$. Then $1=\nr(\alpha \beta^{-1})=p^n p^{2 \ell} u$ for some $u \in R^\times$, and it must be that $n=-2\ell$. But then $\alpha =\begin{pmatrix} 0& p^{-n/2} \\ p^{n/2} &0\end{pmatrix} \mu \beta$. Since both $\alpha$ and $\mu \beta$ are of the form $\begin{pmatrix} a & b\\ p^{n}c & d\end{pmatrix}$ for some $a,b,c,d \in R$, a simple calculation shows we have reached a contradiction since $n>0$. Therefore, $\alpha\beta^{-1} \in F^\times \Or^\times$. A similar reduced norm argument as above shows that in fact, $\alpha \beta^{-1} \in \Or^\times$. This gives $\Or \alpha \beta^{-1}=\Or$ and therefore $\Or \alpha=\Or \beta$, so each Eichler order in the second set corresponds to a well-defined ideal of norm $\p$.
\end{proof}

 Moreover, we have a geometric interpretation of the bijection in Lemma \ref{ideal-order}.  

\begin{prop}\label{segments}
Let $\Or$ be an Eichler order of level $\p^n$ as in Equation (\ref{eichler}), and let $XY$ be the segment associated to $\Or$. Then there is a bijective correspondence \[ \Id(\Or; \p) \overset{\sim}{\longleftrightarrow} \left\{\begin{array}{l}\text{ segments $ZT$ in $\mathcal{T}_\p$ of length $n$} \\ \text{ such that  $\delta(X,Z)=\delta(Y,T)=1$}\end{array}\right\}.\] In particular, each ideal $\Or \alpha$ corresponds to the segment associated to $\alpha^{-1} \Or \alpha$, and which has endpoints $[L_0\alpha]$ and $[L_0\gamma\alpha]$.
\end{prop}

\begin{proof}
We first prove the forward direction. By Lemma \ref{ideal-order}, $\Id(\Or; \p)$ is in correspondence with certain orders conjugate to $\Or$, which then correspond to certain segments of length $n$. 
Let $\alpha^{-1} \Or \alpha$ be such an order; since $\Or$ corresponds to the segment with endpoints $[L_0]$ and $[L_0\gamma]$, then $\alpha^{-1}\Or\alpha$ corresponds to the segment with endpoints $[L_0 \alpha]$ and $ [L_0 \gamma \alpha]$. We claim that $\delta([L_0], [L_0\alpha])=\delta([L_0 \gamma], [L_0 \gamma \alpha])=1$. We need to show that $$\begin{array}{lclcl} p L_0 & \subsetneq & L_0 \alpha & \subsetneq & L_0 \\ p L_0 \gamma & \subsetneq & L_0\gamma  \alpha & \subsetneq & L_0 \gamma\end{array}.$$
Since $\alpha \in \Or=\End_R(L_0) \cap \End_R(L_0 \gamma)$, the containments on the right follow, and since $\alpha$ is not a unit, the containments are strict.  If $\alpha=\begin{pmatrix} a &b \\ p^n c& d\end{pmatrix}$ for $a,b,c,d \in R$, and $\nr(\alpha)=\det(\alpha)=p$, then $p \alpha^{-1}=\begin{pmatrix} d & -b \\ -p^n c & a\end{pmatrix} \in M_2(R)$, so in particular $L_0 p \alpha^{-1} \subsetneq L_0$ which implies the first strict containment. On the other hand, $p L_0 \gamma \subsetneq L_0 \gamma \alpha \iff p L_0 \subsetneq L_0 \gamma \alpha \gamma^{-1}$. But $\gamma \alpha \gamma^{-1} \in \Or$, and $\nr(\gamma \alpha \gamma^{-1})=p$, and as before, we have  $p L_0 \subsetneq L_0 \gamma \alpha \gamma^{-1}$. Therefore, each ideal in $\Id(\Or;\p)$ has a corresponding segment satisfying the conditions in the statement of the proposition.

To prove the backward direction of the correspondence, we use  results of Eichler \cite{ME} and  Brzezinski \cite{JB}, who proved that the number of principal left ideals of reduced norm $\p$ is $2q+1$ when $n=1$, and $2q$ when $n\ge 2$. Therefore, it suffices that to show these are also the number of segments in $\mathcal{T}_\p$ satisfying our conditions.
\begin{figure}

\caption{}
\label{count1proof}

\begin{center}
\begin{tikzpicture}[scale=0.8,
  grow cyclic,
  level distance=2cm,
  level/.style={
    level distance/.expanded=\ifnum#1>1 \tikzleveldistance/1.5\else\tikzleveldistance\fi,
    nodes/.expanded={\ifodd#1 fill=white\else fill=white\fi}
  },
  level 1/.style={sibling angle=12},
  nodes={circle,draw,inner sep=+0pt, minimum size=5pt},
  ]

\draw[dotted] (2,0) -- (4.9,0);

\path[rotate=180]
  node {}
  child foreach \cntI in {1,...,5} {
    node {}
  };

\path[]
  node {}
  child foreach \cntI in {1} {
    node {}
  };

\path[] (0,0) -- (5,0)
  node {}
  child foreach \cntI in {1} {
    node {}
  };

\draw (5.1,0) -- (6.9,0);

\path[] (5,0) -- (7,0)
  node {}
  child foreach \cntI in {1,...,5} {
    node {}
  };


\node[draw=none,fill=none,below] at (0,-0.2)  {\fontsize{8}{6}\selectfont$X$};
\node[draw=none,fill=none,below] at (7,-0.2)  {\fontsize{8}{6}\selectfont$Y$};

\end{tikzpicture}
\end{center}
\end{figure}

 Consider the subtree in Figure \ref{count1proof}, where  $\Or$ corresponds to the  segment $XY$. There are at most three ways we can obtain a segment of length $n$ with vertices $Z$ and $T$ at distance $1$ of $X$ and $Y$. We could either move $XY$ to the left, to the right, and if $n=1$, we could also invert it since $\delta(X,Y)=\delta(Y,X)=n$.

The last statement follows from the bijection $\Or\alpha \leftrightarrow \alpha^{-1}\Or\alpha$.\end{proof}

Proposition \ref{segments} gives rise to the following corollaries, which will help us understand the action of the metacommutation permutation.

\begin{corollary}
\label{extendedaction}
For any $\omega \in \Or^\times$, $\sigma_\omega$ extends to an action on the segments corresponding to $\Id(\Or;\p)$. In particular, if $P=\Or \alpha$, then  $\sigma_\omega(P)$  corresponds to the segment associated to $\omega^{-1} \alpha^{-1}\Or \alpha \omega$.
\end{corollary}

\begin{corollary}
\label{rad}
Let $\rad(\Or)$ be the Jacobson radical of $\Or$, and suppose $\rad(\Or) \in \Id(\Or;\p)$. Then $\rad(\Or)$ corresponds to the segment associated to $\Or$, and is fixed by  $\sigma_\omega$ for any $\omega \in \Or^\times$.
\end{corollary}

\begin{proof}
We have that $\rad(\Or)=\begin{pmatrix} \p &R \\ \p^n &\p\end{pmatrix}$ (for example, see \cite[Remark II.4]{WP}). Then $\rad(\Or) \in \Id(\Or;\p)$ if and only if $d(\Or)=\p$, in which case $\rad(\Or)=\Or\gamma$, and $\rad(\Or)$ corresponds to the segment associated to $\gamma^{-1}\Or\gamma=\Or$. The last statement follows from Lemma \ref{wfix}.
\end{proof}

\section{Results}

In this section, we state our main results on the cycle structure of the permutation given by metacommutation. 
By Corollary \ref{rad}, when $n=1$ the ideal $\rad(\Or)$ is fixed by any $\sigma_\omega$, so it suffices to consider only $2q$ ideals in $\Id(\Or;\p)$. Define \[\Id(\Or;\p)'\colonequals \left\{
\begin{array}{ll}
     \Id(\Or;\p)\setminus\rad(\Or) & \text{if} \,\, n=1 \\
      \Id(\Or;\p) & \text{if} \,\, n\ge 2
\end{array}. \right. \] Moreover,  for each $s\in R/\p$, choose $b_s \in R$ to be any representative of the coset $s+\p \in R/\p$, and define $\alpha_s \colonequals \left(
\begin{array}{cc}
1 & b_s\\
0 & p 
\end{array}\right)$. Consider the sets of ideals  $$S_1\colonequals \{\Or\alpha_s: s \in R/\p\} \quad \text{and} \quad S_2\colonequals  \{\Or\gamma^{-1}\alpha_s\gamma: s \in R/\p\}.$$

 \begin{corollary}
\label{ideals-gens}
         Let $\Or$ be an Eichler order as in  Equation (\ref{eichler}). Then $\Id(\Or;\p)'=S_1 \bigsqcup S_2$.
       \end{corollary}
       
         \begin{proof} 
 First, note that  if two elements in $\Or$ generate different left ideals of $\M_2(R)$, then they generate different left ideals of $\Or$, since  $\Or \alpha=\Or\beta \implies \alpha\beta^{-1} \in \Or^\times \implies \alpha \beta^{-1} \in \GL_2(R) \implies \M_2(R)\alpha=\M_2(R)\beta.$ In particular,  the set  $\{\alpha_s : s\in R/\p\}$ generates $q$ distinct ideals of $\Or$ of reduced norm $\p$.

Now consider  Figure  \ref{generators}. Suppose $XY$ is the segment associated to $\Or$, where $X$ corresponds to the lattice $[L_0]$ and $Y$ to $[L_0\gamma]$. We associate to an ideal $\Or\alpha_s \in S_1$  the segment with endpoints $[L_0\alpha_s]$ and $[L_0\gamma\alpha_s]$. As in the proof of Proposition \ref{segments},  $[L_0\alpha_s]$ is at distance $1$ of $X$, so it is either a vertex to the left of $X$, or just immediately to the right. But one can check that $\alpha_s\gamma_1^{-1} \not\in M_2(R)$, so $[L_0\alpha_s]\ne [L_0\gamma_1]$.  Therefore,  $[L_0\alpha_s]$ is a vertex to the left of $X$.  By Proposition \ref{segments},  the other endpoint  $[L_0\gamma\alpha_s]$ of the segment associated to $\Or\alpha_s$  must be the vertex immediately to the left of $Y$. In other words,  $\Or\alpha_s$ will correspond to a segment obtained by shifting $XY$ to the left.

\begin{figure}
\caption{}
\label{generators}

\begin{center}
\begin{tikzpicture}[scale=0.8,
  grow cyclic,
  level distance=2cm,
  level/.style={
    level distance/.expanded=\ifnum#1>1 \tikzleveldistance/1.5\else\tikzleveldistance\fi,
    nodes/.expanded={\ifodd#1 fill=white\else fill=white\fi}
  },
  level 1/.style={sibling angle=12},
  nodes={circle,draw,inner sep=+0pt, minimum size=5pt},
  ]

\draw[dotted] (2,0) -- (4.9,0);

\path[rotate=180]
  node {}
  child foreach \cntI in {1,...,5} {
    node {}
  };

\path[]
  node {}
  child foreach \cntI in {1} {
    node {}
  };

\path[] (0,0) -- (5,0)
  node {}
  child foreach \cntI in {1} {
    node {}
  };

\draw (5.1,0) -- (6.9,0);

\path[] (5,0) -- (7,0)
  node {}
  child foreach \cntI in {1,...,5} {
    node {}
  };

\node[draw=none,fill=none,above] at (0,0.2)  {\fontsize{8}{6}\selectfont$X$};
\node[draw=none,fill=none,above] at (7,0.2)  {\fontsize{8}{6}\selectfont$Y$};

\node[draw=none,fill=none,below] at (0,-0.1)  {\fontsize{8}{6}\selectfont$[L_0]$};
\node[draw=none,fill=none,below] at (2,0)  {\fontsize{8}{6}\selectfont$[L_0\gamma_1]$};
\node[draw=none,fill=none,below] at (2,0.1)  {\fontsize{8}{6}\selectfont$=[L_0\gamma^{-1}\alpha_s\gamma]$};
\node[draw=none,fill=none,below] at (7,0)  {\fontsize{8}{6}\selectfont$[L_0\gamma]$};
\node[draw=none,fill=none,below] at (-1.7,-0.7)  {\fontsize{8}{6}\selectfont$[L_0\alpha_s]$};
\node[draw=none,fill=none,below] at (5,0.2)  {\fontsize{8}{6}\selectfont$[L_0\gamma_{n-1}]$};
\node[draw=none,fill=none,below] at (5,-0.2)  {\fontsize{8}{6}\selectfont$=[L_0\gamma\alpha_s]$};
\node[draw=none,fill=none,below] at (9,-0.5)  {\fontsize{8}{6}\selectfont$[L_0\alpha_s\gamma]$};

\end{tikzpicture}
\end{center}
\end{figure}

We shift our attention to $S_2$. Since $\nr(\gamma^{-1}\alpha_s\gamma)=p$ and $\gamma^{-1}\alpha_s\gamma \in \Or^\times$, $S_2$ indeed consists of ideals of norm $\p$. Consider the ideal $\Or\gamma^{-1}\alpha_s\gamma$, to which we associate the segment with endpoints $[L_0 \gamma^{-1}\alpha_s\gamma]$ and $[L_0\alpha_s\gamma]$. We claim that this segment is obtained by shifting $XY$ to the right; in particular, we claim that $\delta([L_0], [L_0\alpha_s\gamma])=n+1$. Note that $\alpha_s\gamma=\begin{pmatrix} b_sp^n&1\\p^{n+1}&0\end{pmatrix}=\begin{pmatrix}1&0\\0& p^{n+1} \end{pmatrix}\begin{pmatrix} b_sp^n&1\\ 1&0\end{pmatrix}$. Since the action of $B^\times$ preserves $\delta$, we have $$\delta([L_0], [L_0\alpha_s\gamma])=\delta\left(\left[L_0\begin{pmatrix} b_sp^n&1\\ 1&0\end{pmatrix}^{-1}\right], \left[L_0\begin{pmatrix}1&0\\0& p^{n+1} \end{pmatrix}\right]\right).$$ Since  $\begin{pmatrix} b_sp^n&1\\ 1&0\end{pmatrix}^{-1} \in \GL_2(R)$, and by the definition of $\delta$, we have $$\delta([L_0], [L_0\alpha_s\gamma])=\delta\left([L_0], \left[L_0\begin{pmatrix}1&0\\0& p^{n+1} \end{pmatrix}\right]\right)=n+1.$$
Finally, we show that each element of the form $\gamma^{-1}\alpha_s\gamma$ generates a distinct ideal.  Note that $\Or \gamma^{-1}\alpha_s\gamma=\Or\gamma^{-1}\alpha_r\gamma$ if and only if $ \gamma^{-1}\alpha_s\alpha_r^{-1}\gamma \in \Or^\times$. On the other hand, $\gamma\Or\gamma^{-1}=\Or$, so it follows that $\alpha_s\alpha_r^{-1} \in \Or^\times$, and a simple calculation shows this happens only if $\alpha_s=\alpha_r$.
           \end{proof}

\begin{example}
\label{example1}
Consider the Eichler order $$\Or=\begin{pmatrix} \Z_3&\Z_3\\ 3\Z_3&\Z_3\end{pmatrix}=M_2(\Z_3)\cap \gamma^{-1} M_2(\Z_3) \gamma=\End_R(L_0)\cap \End(L_0 \gamma)$$ where $\gamma=\begin{pmatrix} 0&1\\ 3 &0\end{pmatrix}$. Then $\Or$ corresponds to the segment between the vertices given by $[L_0]$ and $[L_0 \gamma]$ bolded in Figure \ref{figexample1}. Corollary \ref{ideals-gens} gives $$S_1= \left\{ \Or \left(
\begin{array}{cc}
1 & 0\\
0 & 3
\end{array} \right), \Or \left(
\begin{array}{cc}
1 & 1\\
0 & 3
\end{array} \right), \Or \left(
\begin{array}{cc}
1 & 2\\
0 & 3
\end{array} \right) \right \};$$ $$S_2=\left\{ \Or \left(
\begin{array}{cc}
3 & 0\\
0 & 1
\end{array} \right),\Or \left(
\begin{array}{cc}
3 & 0\\
3 & 1
\end{array} \right), \Or \left(
\begin{array}{cc}
3 & 0\\
6 & 1
\end{array} \right) \right\}.$$

We obtain the segment associated to $\Or \left(
\begin{array}{cc}
3 & 0\\
3 & 1
\end{array} \right)=\Or\gamma^{-1}\alpha_1\gamma$ from the action of $\gamma^{-1}\alpha_1\gamma$ on $[L_0]$ and $[L_0\gamma]$, and we get  the endpoints  $[L_0\gamma]$ and $[L_0\alpha_1\gamma]$. Indeed, 
\[[L_0\gamma^{-1}\alpha_1\gamma]=\left[L_0  \left(
\begin{array}{cc}
3 & 0\\
3 & 1
\end{array} \right)\right]=\left[L_0 \left(
\begin{array}{cc}
0 & 1\\
3 & 0
\end{array} \right)\right]=[L_0\gamma],\] where the equalities are coming from the fact that elementary row operations are units in $\M_2(R)=\End_R(L_0)$. The segment associated to $\Or \gamma^{-1}\alpha_1\gamma$ is the dashed segment in Figure \ref{figexample1}.

\end{example}

\begin{figure}
\caption{}
\label{figexample1}
\begin{center}
\begin{tikzpicture}[scale=0.8, every node/.style={circle, draw,fill=white, inner sep=1.8pt}]

\draw (0,0) -- (0,3);
\draw (0,0) -- (0,-3);
\draw[line width=0.5mm] (0,0) -- (3,0);
\draw (0,0) -- (-3,0);
\draw[dashed] (3,0) -- (3,3);
\draw (3,0) -- (3,-3);
\draw (3,0) -- (6,0);

\draw (0,3) -- (0,3.7);
\draw (0,3) -- (0.7,3);
\draw (0,3) -- (-0.7,3);
\draw (-3,0) -- (-3,0.7);
\draw (-3,0) -- (-3,-0.7);
\draw (-3,0) -- (-3.7,0);
\draw (0,-3) -- (0,-3.7);
\draw (0,-3) -- (0.7,-3);
\draw (0,-3) -- (-0.7,-3);

\draw (3,3) -- (3,3.7);
\draw (3,3) -- (3.7,3);
\draw (3,3) -- (2.3,3);
\draw (6,0) -- (6,0.7);
\draw (6,0) -- (6,-0.7);
\draw (6,0) -- (6.7,0);
\draw (3,-3) -- (3,-3.7);
\draw (3,-3) -- (3.7,-3);
\draw (3,-3) -- (2.3,-3);

\node at (0,0) {}; \node at (3,0){}; \node at (-3,0){}; \node at (3,3){}; \node at (3,-3){}; \node at (6,0){}; \node at (0,3){}; \node at (0,-3){};

\node[draw=none,fill=none] at (0.2,-0.3)  {\fontsize{8}{6}\selectfont$\qquad [L_0]$};
\node[draw=none,fill=none] at (3.3,-0.3)  {\fontsize{8}{6}\selectfont$\qquad [L_0\gamma ]$};
\node[draw=none,fill=none] at (0.4,2.7)  {\fontsize{8}{6}\selectfont$\qquad [L_0\alpha_1]$};
\node[draw=none,fill=none] at (3.5,2.7)  {\fontsize{8}{6}\selectfont$\qquad [L_0\alpha_1\gamma]$};
\node[draw=none,fill=none] at (-2.6,-0.3)  {\fontsize{8}{6}\selectfont$\qquad [L_0\alpha_0]$};
\node[draw=none,fill=none] at (6.5,-0.3)  {\fontsize{8}{6}\selectfont$\qquad [L_0\alpha_0\gamma ]$};
\node[draw=none,fill=none] at (0.4,-2.7)  {\fontsize{8}{6}\selectfont$\qquad [L_0\alpha_2$]};
\node[draw=none,fill=none] at (3.5,-2.7)  {\fontsize{8}{6}\selectfont$\qquad [L_0\alpha_2\gamma]$};
\end{tikzpicture}
\end{center}
\end{figure}

\begin{remark}
\label{bijecS1S2} 
Conjugation by $\gamma$ induces a bijection between the sets $S_1$ and $S_2$. Note that since $\gamma^{-1}\Or\gamma=\Or$, conjugating the generators coincides with conjugating the ideals, i.e. $\gamma^{-1}\Or\alpha_s\gamma=\Or\gamma^{-1}\alpha_s\gamma$.
\end{remark}

Define $\phi_\gamma\colon S_2 \rightarrow S_1$ by $$\phi_\gamma( \Or\gamma^{-1}\alpha_s\gamma) \colonequals  \Or\alpha_s.$$ 

We may suspect that $\sigma_\omega$ will fix each of the two sets of ideals determined by $S_1$ and $S_2$. We confirm this in the following corollary.

\begin{lemma}\label{lem4} The sets $S_1$ and $S_2$ are fixed under the permutation $\sigma_\omega$.
\end{lemma}

            \begin{proof}
            Refer back to Figure \ref{generators}, where $XY$ is the segment associated to $\Or$ with endpoints $[L_0]$ and $[L_0\gamma]$. Consider the ideal  $\Or \alpha_s$. Since we obtain the segment associated to $\Or\alpha_s$ by shifting $XY$ one unit to the left, this segment has endpoints $[L_0\alpha_s]$ and $[L_0\gamma_{n-1}]$. By Lemma \ref{wfix}, $\omega \in \Or^\times$ acts on the tree $\mathcal{T}_\p$ by fixing all the vertices contained on $XY$, and therefore all the vertices on the segment associated to $\Or\alpha_s$ besides $[L_0\alpha_s]$. Since this action preserves $\delta$, $[L_0\alpha_s\omega]$ must be at  distance $1$ of $[L_0]$, and the only allowed choices for $[L_0\alpha_s\omega]$ are the set $\{[L_0\alpha_r]\}$. 

 The statement for $S_2$ follows analogously, this time associating to the ideal $\Or\gamma^{-1}\alpha_s\gamma$ the segment  obtained by shifting $XY$ to the right. 
          \end{proof}

This means that we may view $\sigma_\omega \in \Sym(S_1) \times \Sym(S_2)$.

Consider the set of ideals of $\M_2(R)$ given by \[\Id(\M_2(R);\p)'\colonequals \{\M_2(R)\alpha_s\} =  \Id(\M_2(R);\p) \setminus \{M_2(R)\gamma_1\}.\]  Define the permutation $\tau_\omega$ by $\tau_\omega(P)=P\omega$.  Note that $\M_2(R)\gamma_1$ corresponds to the vertex $[L_0\gamma_1]$, which is to the immediate right of $X$. By Lemma \ref{wfix}, any $\omega \in \Or^\times$ fixes $[L_0\gamma_1]$, so $\tau_\omega$  gives a permutation of $\Id(\M_2(R);\p)'$.

 The ideals in $S_1$ exactly correspond to those in $\Id(\M_2(R);\p)'$ via the bijection \[ \begin{array}{clll}\varphi \colon & S_1 & \rightarrow & \Id(\M_2(R);\p)' \\ & \Or \alpha_s & \mapsto &\M_2(R)\alpha_s \end{array}.\]

\begin{theorem}\label{thm4}
We may understand the permutation $\sigma_\omega$ by computing $\sigma_\omega|_{S_1}$ and $\sigma_\omega|_{S_2}$ separately, via the following diagrams, which commute.
\begin{enumerate}

\item[\rm{a.}] $$\xymatrix{
S_2 \ar^{\sigma_\omega|_{S_2}\hspace{.1in}}[r]\ar_{\phi_{\gamma}}@{^{}->}[d]&{S_2}\ar^{\phi_{\gamma}}[d]\\
{S_1}\ar^{\sigma_{\gamma^{-1}\omega \gamma}|_{S_1}\hspace{0in}}[r]&{S_1}
}$$

\item[\rm{b.}] $$\xymatrix{
S_1 \ar^{\sigma_\omega|_{S_1}\hspace{.1in}}[r]\ar_{\varphi}@{^{}->}[d]&{S_1}\ar^{\varphi}[d]\\
{\Id(\M_2(R);\p)'}\ar^{\tau_\omega\hspace{0in}}[r]&{\Id(\M_2(R);\p)'}
}$$

\item[\rm{c.}] $$ \xymatrix{
S_2 \ar^{\sigma_\omega|_{S_2}\hspace{.1in}}[r]\ar_{\varphi \circ \phi_{\gamma}}@{^{}->}[d]&{S_2}\ar^{\varphi \circ \phi_{\gamma}}[d]\\
{\Id(\M_2(R);\p)'}\ar^{\tau_{\gamma^{-1}\omega\gamma}\hspace{0in}}[r]&{\Id(\M_2(R);\p)'}
}$$

\end{enumerate}
\end{theorem}

\begin{proof}

To show that (a) commutes, let $\Or\beta \in S_2$. Then, $\phi_{\gamma_n} \circ \sigma_\omega(\Or\beta)=\phi_{\gamma}(\Or\beta \omega)=\Or\gamma^{-1}\beta \omega \gamma$ and $\sigma_{\gamma^{-1}\omega\gamma} \circ \phi_{\gamma}(\Or\beta)=\sigma_{\gamma^{-1}\omega\gamma} (\Or\gamma^{-1}\beta \gamma)=\Or\gamma^{-1}\beta \gamma\gamma^{-1}\omega \gamma=\Or\gamma^{-1}\beta \omega \gamma=\phi_{\gamma^{-1}\omega\gamma} \circ \sigma_\omega(\Or\beta)$.

For (b), let $\Or \alpha \in S_1$. Then, $\varphi \circ \sigma_\omega(\Or\alpha)=\M_2(R)\alpha \omega$, and $\tau_\omega \circ \varphi(\Or\alpha)=\M_2(R)\alpha_\omega$. 

The diagram in (c) is a composition of the first diagram and that in Theorem \ref{thm4}, so commutativity follows. \end{proof}

This means that locally, we may write the cycle structure of $\sigma_\omega$ in an Eichler order in terms of the cycle structure of two separate permutations given by metacommutation in $\M_2(R)$.

We may then define the following maps:  

\begin{align*} \sigma \colon \Or^\times &\rightarrow \Sym(\Id(\Or;\p)')\\
\omega &\mapsto \sigma_\omega
\end{align*}

\begin{align*}
\tau \times \tau^\gamma \colon \Or^\times & \rightarrow \Sym(\Id(\M_2(R)); \p) \times \Sym(\Id(\M_2(R)); \p)\\
\omega & \mapsto (\tau_\omega, \tau_{\gamma^{-1}\omega \gamma}).
\end{align*}

Then, Theorem \ref{thm4} is summarized by the following commutative diagram, which gives a complete description of $\sigma_\omega$ in terms of the corresponding cycle structures in the maximal order $\M_2(R)$, which is known (see Forsyth--Gurev--Shrima \cite{FGS} or Chari \cite{SC}). 

\begin{corollary}\label{cor1} 
There is an embedding $$\iota \colon \Sym(\Id(\M_2(R)'); \p) \times \Sym(\Id(\M_2(R)'); \p) \rightarrow \Sym(\Id(\Or;\p)')$$ such that $$\sigma(\omega)=\iota \circ (\tau \times \tau^{(\gamma)})(\omega);$$  i.e., the following diagram commutes:

$$\xymatrix@R=8pt{
{}&{\Sym(\Id(\M_2(R);\p)')\times \Sym(\Id(\M_2(R);\p)')}\ar_{\iota}@{^{(}->}[dd]\\
{\Or^\times}\ar_{\tau \times \tau^{(\gamma)}}[ur]\ar^{\sigma}[dr]\\
&{\Sym(\Id(\Or;\p)')}
}$$
In other words, 

\begin{enumerate}
\item $\omega$ permutes the elements of $S_1$ by $$\tau_\omega(\M_2(R)\alpha)=\sigma_\omega(\Or\alpha)$$ as in the maximal case.
\item $\omega$ permutes the elements of $S_2$ by $$\tau_\omega(\M_2(R)\gamma^{-1}\alpha\gamma)=\gamma^{-1}\sigma_{\gamma^{-1}\omega\gamma}(\Or\alpha)\gamma$$ again as in the maximal case, since $\Or\gamma^{-1}\alpha \gamma \in S_2$.
\end{enumerate}

%

\end{corollary}
%
%
%
\begin{example}\label{ex2}

Let $n=1$, $p=3$, and choose $\omega=\left(
\begin{array}{cc}
1 & 1\\
0 & 1 
\end{array} \right)$, so $\gamma_1^{-1}\omega\gamma_1=\left(
\begin{array}{cc}
1 & 0\\
3 & 1
\end{array} \right)$. Using Theorem \ref{thm2}, we have
 \[
\begin{array}{rrr} \sigma_\omega(\Or\alpha_0)=\Or \alpha_1, &\sigma_\omega(\Or \alpha_1)=\alpha_2,& \sigma_\omega(\Or \alpha_2)=\Or \alpha_1, \\ 
 \sigma_{\gamma_1^{-1}\omega\gamma}(\Or\alpha_0)=\Or \alpha_0,& \sigma_{\gamma_1^{-1}\omega\gamma}(\Or \alpha_1)=\alpha_1, & \sigma_{\gamma_1^{-1}\omega\gamma}(\Or \alpha_2)=\Or \alpha_2.\end{array}
\] 
 Via the correspondence of $S_1$ and $S_2$ in Remark \ref{bijecS1S2} and Theorem \ref{thm4}(c), one may verify that we get $$\sigma_\omega(\Or\beta_1)=\phi_{\gamma_1}^{-1}(\sigma_{\gamma^{-1}\omega\gamma_1}(\Or\alpha_1))=\phi_{\gamma_1}^{-1}(\Or\alpha_1)=\Or\beta_1;$$
$$\sigma_\omega(\Or\beta_2)=\phi_{\gamma_1}^{-1}(\sigma_{\gamma_1^{-1}\omega\gamma}(\Or\alpha_2))=\phi_{\gamma_1}^{-1}(\Or\alpha_2)=\Or\beta_2;$$ $$\sigma_\omega(\Or\beta^0)=\phi_{\gamma_1}^{-1}(\sigma_{\gamma_1^{-1}\omega\gamma}(\Or\alpha_0))=\phi_{\gamma_1}^{-1}(\Or\alpha_0)=\Or\beta^0.$$
In other words, $\sigma_\omega$ permutes the ideals in $S_1$ by a $3$-cycle, and fixes each ideal in $S_2$.
\end{example}

\section{Conclusion}

We conclude with a more detailed description of the cycle structure of $\sigma_\omega$. Again, suppose that $\Or$ is an Eichler order of level $\p^n$ with $\omega \in \Or^\times$.  We now have a description of the cycle structures of the permutation $\sigma_{\omega}$ on $S_1$ and $S_2$, and hence on $\Id(\Or;\p)$ in terms of the corresponding cycle structures of the permutations $\tau_\omega$ and $\tau_{\gamma^{-1}\omega\gamma}$ of $\Id(\M_2(R);\p)'$. We conclude with a brief discussion of the image and kernel of the map $\sigma$ in Corollary \ref{cor1}. We begin with a lemma about the relationship between $\omega$ and $\gamma^{-1}\omega\gamma$. 

\begin{lemma}\label{lem5}
If $\omega \in \Or^\times$ with $\omega \equiv a \pmod \p$ is a scalar matrix $\pmod \p$,  then $\gamma^{-1}\omega\gamma \equiv \left(
\begin{array}{cc}
a & c\\
0 & a
\end{array} \right) \pmod \p$, where $c$ is the coefficient of $p^n$ in the bottom left entry of $\omega$.
\end{lemma}

\begin{proof}

If $\omega \equiv a \pmod \p$, then we may write $\omega=\left(
\begin{array}{cc}
a+k_1p & bp\\
cp^n & a+k_2p
\end{array} \right)$. A quick computation shows that $\gamma^{-1}\omega \gamma=\left(
\begin{array}{cc}
a+k_1p & c\\
bp^{n+1} & a+k_2p
\end{array} \right) \equiv \left(
\begin{array}{cc}
a & c\\
0 & a 
\end{array} \right) \pmod \p$.
\end{proof}

Now, by Forsyth--Gurev--Shrima \cite{FGS}, all cycles of $\tau_\omega$ are the same size, and the length is the multiplicative order of $\omega+\p$ in $\PGL_2(R/\p)$. Since $\sigma_\omega|_{S_1}$ has the same cycle structure as $\tau_{\omega}$ and $\sigma_\omega|_{S_2}$ has the same cycle structure of $\tau_{\gamma^{-1}\omega\gamma}$, it is of interest to compare the cycle structure of $\tau_{\omega}$ and $\tau_{\gamma^{-1}\omega\gamma}$. In other words, we seek to compare the multiplicative orders of $\omega+\p$ and $\gamma^{-1}\omega\gamma+\p$ in $\PGL_2(R/\p)$. 

\begin{lemma}\label{lem6}
The permutations $\tau_\omega$ and $\tau_{\gamma^{-1}\omega\gamma}$ have the same number of fixed points, unless one of $\omega$ and $\gamma^{-1}\omega\gamma$ is a scalar matrix $\pmod \p$. Moreover, they each have at most $1$ fixed point when neither is a scalar $\pmod \p$. 
\end{lemma}
\begin{proof}
By modifying the equation for the number of fixed points given by Cohn--Kumar \cite{CohnKumar} and Chari \cite{SC} to exclude the element $\M_2(R)\gamma_1$ (which must also be fixed under $\tau_\omega$ following Corollary \ref{rad}), if $\omega$ is not a scalar matrix $\pmod \p$, the number of fixed points of $\tau_\omega$ is given by the Legendre symbol $\left(\frac{\trd(\omega)^2-4\nrd(\omega)}{p}\right)$. Writing $\omega=\left(
\begin{array}{cc}
a & b\\
p^nc & d
\end{array} \right)$, we have $(a+d)^2-4ad-bp^nc \equiv (a-d)^2\pmod \p$, so $\left(\frac{\trd(\omega)^2-4\nrd(\omega)}{p}\right)=0$ or $1$, depending on whether $a \equiv d \pmod \p$ or not. But, since the trace and norm maps are invariant under conjugation (say, by $\gamma_n$), the expression is the same for $\omega$ and $\gamma^{-1}\omega\gamma$. 
\end{proof}

\begin{remark} See Example \ref{ex2} for a counterexample in the case where $\gamma^{-1}\omega\gamma$ is a scalar matrix $\pmod \p$.
\end{remark}

We now discuss the kernel of the map $\sigma$ in Corollary \ref{cor1}.

\begin{theorem}
The kernel of the map $\sigma$ is $$\ker(\sigma)=\left \{ \left(
\begin{array}{cc}
a+k_1p &bp\\
cp^{n+1} & a+k_2p
\end{array} \right):  a,b,c,k_1, k_2 \in R \right \}.$$
\end{theorem}

\begin{proof}
First, $\tau_\omega$ is the identity permutation if and only if $\omega \equiv \alpha \pmod \p$ is a scalar $\pmod \p$. Therefore, $\sigma_\omega$ is the identity permutation if and only if $\tau_\omega$ and $\tau_{\gamma^{-1}\omega\gamma}$ are both the identity permutation, if and only if $\omega$ and $\gamma^{-1}\omega\gamma$ are both scalar matrices $\pmod \p$. Now, by Lemma \ref{lem5}, 
this holds if and only if $\omega=\left(
\begin{array}{cc}
a+k_1p & bp\\
cp^{n+1} & a+k_2p
\end{array} \right)$ for $a,b, c, k_1, k_2 \in R$.\end{proof}

Finally, we discuss the image of the map $\sigma$ in Corollary \ref{cor1}.

\begin{theorem}
Let $\ell_1>1$ and $\ell_2>1$ denote the size of the cycles that aren't fixed points of $\sigma_{\omega}|_{S_1}$ and $\sigma_{\omega}|_{S_2}$, respectively. Then, if $q=\opchar(\F_q)$ is prime or if $\opchar(\F_2) \nmid \trd(\omega)^2-4\nrd(\omega)$, then $\ell_1 = \ell_2$. 
\end{theorem}

\begin{proof} By Lemma \ref{lem5}, since $\omega^{\ell_1} \equiv a \pmod \p$ for some $a \in R$,  we also have ${\gamma^{-1}\omega\gamma}^{\opchar(\F_q)\ell_1}\equiv a \pmod \p$. Then, we must have $\ell_2 \mid \opchar(\F_q)\ell_1$. For the same reason, $\ell_1 \mid p\ell_2$. If $q=\opchar(\F_q)$ is prime, then $\ell_1 \leq \opchar(\F_q)$ and $\ell_2 \leq \opchar(\F_q)$ since the size of the cycles must be smaller than the set being permuted, so either $\ell_1=\ell_2=\opchar(\F_q)$ or $\ell_1 \mid \ell_2$ and $\ell_2 \mid \ell_1$ so equality holds.

If $\p \nmid (\trd(\omega)^2-4\nrd(\omega))$, then by Cohn--Kumar \cite{CohnKumar} and Chari \cite{SC}, $\tau_\omega$ has two fixed points, meaning there are $q-1$ remaining in each $\Id(\M_2(R);\p)$ (and hence in $S_1$ and $S_2$) to be permuted, so $\ell_1 \mid (q-1)$ and $\ell_2 \mid (q-1)$, so since $\gcd(\opchar(\F_q),q-1)=1$, we must have $\ell_1 \mid \ell_2$ and $\ell_2 \mid \ell_1$ so equality holds.

\end{proof}

\begin{remark}
The requirement that $\p \mid (\trd(\omega)^2-4\nrd(\omega))$ is equivalent to saying that the diagonal entries are equivalent $\pmod \p$.
\end{remark}

%

\end{document}